\documentclass[12pt, reqno]{amsart}
\usepackage{amsmath, amssymb, amsthm, fullpage}

\newtheorem{thm}{Theorem}

\newtheorem{lemma}[thm]{Lemma}
\newtheorem{theorem}[thm]{Theorem}

\numberwithin{equation}{section}

\theoremstyle{definition}
\newtheorem{rem}[thm]{Remark}

\newcommand{\al}{\alpha}
\renewcommand{\b}{\beta}

\newcommand{\e}{\varepsilon}
\newcommand{\la}{\lambda}
\renewcommand{\phi}{\varphi}

\newcommand{\R}{{\mathbb R}}
\newcommand{\br}[1]{\left\langle #1 \right\rangle}

\newcommand{\Step}[1]{\noindent \textbf{Step #1}}
\newcommand{\so}{\quad\Longrightarrow\quad}
\newcommand{\nequiv}{\mathrel{\setbox0\hbox{$\equiv$}%
                     \rlap{\hbox{$\equiv$}}\hbox to \wd0{\hfil $/$\hfil}}}

\renewcommand{\qed}{\rule{3mm}{3mm}}
\renewenvironment{proof}
    {\vspace{1mm}\noindent\textbf{Proof.}}
    {\hspace*{\fill} $\qed$\vspace{1mm}}
\newenvironment{proof_of}[1]
    {\vspace{1mm}\noindent {\bf Proof of #1.}}
    {\hspace*{\fill} $\qed$\vspace{1mm}}

\begin{document}
\title[Existence of ground states]{Existence of ground states for fourth-order wave equations}
\author{Paschalis Karageorgis}
\address{School of Mathematics, Trinity College, Dublin 2, Ireland}
\email{pete@maths.tcd.ie}
\author{P.J. McKenna}
\address{University of Connecticut, Storrs, CT 06269}
\email{mckenna@math.uconn.edu}

\begin{abstract}
Focusing on the fourth-order wave equation $u_{tt} + \Delta^2 u + f(u)= 0$, we prove the existence of ground state solutions $u=
u(x+ct)$ for an optimal range of speeds $c\in\R^n$ and a variety of nonlinearities $f$.

\end{abstract}
\maketitle

\section{Introduction}
There is now a substantial literature on traveling waves in nonlinearly supported beams.  The first of these papers \cite{mckwa} was
inspired by an old report of the existence of traveling waves on the Golden Gate Bridge in San Francisco \cite{tacoma}.  There are now
results on three main types of restoring forces for the equation
\begin{equation}
u_{tt} + u_{xxxx} +f(u) =0, \quad\quad x\in\R.
\end{equation}

The first type of nonlinearity that was studied was a piecewise linear one, reflecting the fact that when cables loose tension, they
do not resist compression; this led in \cite{mckwa} to the study of the equation
\begin{equation}\label{pdebeam}
u_{tt}+u_{xxxx} + u^+ =1, \quad\quad x\in\R.
\end{equation}
Solutions of the form $u(x,t)= 1+y(x-ct)$ were found by reducing the partial differential equation (\ref{pdebeam}) to the ordinary
differential equation on the real line
\begin{equation*}
y^{iv}+c^2y''+(1+y)^+=1
\end{equation*}
and then solving explicitly the two linear equations $y^{iv}+c^2y''+y=0$, where $y\geq -1$, and $y^{iv}+c^2y''=1$, where $y\leq -1$.
Solutions of both equations were constructed which matched at the boundary $y=-1$ and which tended to zero exponentially as $|x|\to
\infty$ by showing that solutions corresponded to zeros of a certain transcendental function. These zeros were then found numerically
when $0<c<\sqrt2$.

This paper had some shortcomings: first, it was technically not a proof, since no account was taken of factors like roundoff when
finding the zeros of the transcendental function. Second, it did not prove existence in the natural range of wave speeds, namely $0<
c< \sqrt{2} $. Third, there was no effort in the paper to get information on the stability of the many solutions that were calculated.
The first flaw was remedied in \cite{CheMcK}, where a rigorous proof of the existence of solutions of \eqref{pdebeam} was given using
the mountain pass theorem and the method of concentrated compactness.

The second type of nonlinearity that was studied was a "smoothed version" of the piecewise linear one, namely $f(u)= e^u-1$. The
reason for this substitution was that the piecewise linear nonlinearity was not amenable to accurate solving of the initial value
problem in order to determine stability properties of the waves shown or calculated to exist.  This substitution led to the discovery
of a large class of traveling wave solutions with extraordinary interaction and fission properties which remain unexplained to date
\cite{CheMcK, homck}. Initially, solutions were calculated using the mountain pass algorithm \cite{choimck}, however later it became
clear that shooting methods were a faster and more efficient substitute \cite{chamckze}.

Although the substitution of the new nonlinearity gave rise to some beautiful numerical results, it also introduced a new problem; the
existence of homoclinic solutions of
\begin{equation} \label{beam}
y^{iv}+c^2y^{\prime\prime}+ e^y-1=0
\end{equation}
was not apparent and was left as an open problem. While the mountain pass algorithm converged to traveling wave solutions, key
estimates in the proof for the piecewise smooth nonlinearity did not go through.  In fact, the existence of solutions of \eqref{beam}
still remains open, even though there has been some progress.  Smets and van den Berg \cite{smva} showed that for {\it almost all} $c$
in the interval $(0,\sqrt{2})$, there exists at least one solution.  Later, in \cite{plum}, the particular case $c=1.3$ was studied
and at least $36$ solutions were shown to exist using a computer-assisted proof.

The third type of nonlinearity was studied by Levandosky \cite{lev2,lev1} who focused on the case $f(u)=u-|u|^{q-1}u$ under suitable
restrictions on $q$.  Here, the $n$-dimensional problem
\begin{equation*}
u_{tt}+ \Delta^2 u +u = |u|^{q-1}u
\end{equation*}
was considered for any dimension $n\geq 1$ and \textit{ground state} solutions were shown to exist by virtue of a constrained
minimization technique.  This was the first higher-dimensional result and it also provided information on the stability of traveling
waves; some higher-dimensional numerical results were obtained in \cite{homck} for the exponential type nonlinearity.

In this paper, we give a new approach which is inspired by recent work of Liu \cite{Yue} on the Ostrovsky equation.  More precisely,
we adopt the Nehari manifold approach \cite{Willem}, commonly used for second-order problems, and we construct ground state solutions
of
\begin{equation} \label{4}
u_{tt} + \Delta^2 u + u + f(u)=0, \quad\quad x\in\R^n
\end{equation}
for a large class of nonlinearities $f$.  The second-order analogue of \eqref{4} is easier because the corresponding ground states are
radial, so the compactness lemma of Strauss \cite{Walter} becomes applicable; unfortunately, this is no longer the case for the
fourth-order problem.  To prove existence of ground states for \eqref{4}, we shall now extend the method of Liu \cite{Yue} and treat a
general class of polynomially growing nonlinear terms.

Our precise assumptions for the nonlinear term $f$ in \eqref{4} are the following.
\begin{itemize}
\item[(A1)]
$f$ is differentiable almost everywhere with
\begin{equation}\label{grow}
\liminf\limits_{u\to-\infty} f(u) +u > -\infty.
\end{equation}
\item[(A2)]
There exist constants $C_1>0$, $p\geq 1$ and $q>1$ such that
\begin{alignat}{2}
|f(u)-f(v)| &\leq C_1(|u|+|v|)^{p-1} \,|u-v| \quad &&\text{for all $u,v\in\R$;} \label{f1} \\
|f(u)|&\leq C_1|u|^q \quad &&\text{for all $u\in\R$.} \label{f2}
\end{alignat}
If $n>4$, then we also assume $p,q<\frac{n+4}{n-4}$ so that $H^2\subset L^{p+1}\cap L^{q+1}$ in any case.
\item[(A3)]
There exists $\mu\geq 2$ such that $G(u)= \int_0^u f(s)\,ds -\frac{1}{\mu} \,uf(u)$ is non-negative and convex.  If $\mu=2$, then we
also assume $G(u)\geq C_2|f(u)|$ for some $C_2>0$ and all $u\in\R$.
\item[(A4)]
$u^2f'(u) - uf(u)\leq 0$ for almost all $u\in\R$ with equality if and only if $f(u)=0$.
\end{itemize}

\begin{rem}
Our assumption (A4) holds for any function $f$ which is convex for $u<0$ and concave for $u>0$, as long as the convexity/concavity is
strict at points where $f$ is nonzero.  Our assumption (A3) is a variant of the Ambrosetti-Rabinowitz condition \cite{AR} which arises
in the study of the second-order analogue of \eqref{4} but requires that $\mu>2$.
\end{rem}

\begin{rem}
Two typical functions which satisfy all our assumptions are
\begin{equation}\label{exa}
f(u)= -|u|^{q-1}u, \quad\quad f(u) = -\min(u+1,0),
\end{equation}
where $1<q<\frac{n+4}{n-4}$.  In the former case, (A3) holds with $\mu= q+1>2$ and $G(u)$ is identically zero, so our analysis
simplifies quite a bit; this case was studied in \cite{lev2, lev1}.  The piecewise linear case was studied in \cite{CheMcK, mckwa} and
is surprisingly more difficult; in that case, (A3) holds with $\mu=2$ and the functional \eqref{Jc} we need to minimize no longer
bounds the $H^2$-norm.  In fact, we can only treat the case $\mu=2$ in the physically relevant dimensions $n<4$.
\end{rem}

The main result of this paper can now be stated as follows.

\begin{theorem}\label{exi}
Assume (A1)-(A4) and that $0<|c|<\sqrt2$.
\begin{itemize}
\item[(a)]
If $\mu>2$, then equation \eqref{4} has ground state solutions for any $n\geq 1$.
\item[(b)]
If $\mu=2$, then equation \eqref{4} has ground state solutions when $n=1,2,3$.
\end{itemize}
\end{theorem}

In section \ref{mipr}, we introduce the action \eqref{Ic} whose critical points correspond to traveling wave solutions of \eqref{4}
and we set up a minimization problem over the associated Nehari manifold \eqref{dc2}.   The proof of Theorem \ref{exi} is given in
section \ref{pomr}, where we also characterize the ground states in terms of the action; see Theorem \ref{new}.

\section{The minimization problem}\label{mipr}
In this section, we look for traveling wave solutions of the fourth-order equation
\begin{equation*}
u_{tt} + \Delta^2 u + u + f(u)=0, \quad\quad x\in\R^n.
\end{equation*}
To say that $u(x,t)= \phi(x+ct)$ is a solution for some $c\in\R^n$ is to say that
\begin{equation}\label{pde}
\Delta^2 \phi + \sum_{i,j=1}^n c_i c_j \phi_{x_ix_j} + \phi + f(\phi) = 0.
\end{equation}
To construct solutions of \eqref{pde}, we shall now look for critical points of the functional
\begin{equation}\label{Ic}
I_c(z) = \frac{1}{2} \int_{\R^n} (\Delta z)^2 - (c\cdot \nabla z)^2 + z^2 \:dx + \int_{\R^n} F(z) \:dx,
\end{equation}
where $F(u)= \int_0^u f(s)\,ds$.  Since the Fr\'echet derivative of this functional is given by
\begin{equation*}
\br{I_c'(z),\phi} = \int_{\R^n} \Delta z \Delta\phi - (c\cdot \nabla z)(c\cdot \nabla \phi) + z\phi \:dx + \int_{\R^n}
f(z)\phi \:dx,
\end{equation*}
it is clear that every critical point of $I_c$ is also a root of
\begin{equation}\label{Pc}
P_c(z) = \br{I_c'(z),z} = \int_{\R^n} (\Delta z)^2 - (c\cdot \nabla z)^2 + z^2 \:dx + \int_{\R^n} zf(z) \:dx.
\end{equation}
Our goal is to reduce the existence of solutions of \eqref{pde} to the existence of minimizers for
\begin{equation}\label{dc2}
d_c = \inf \{ I_c(z) : 0\nequiv z\in H^2(\R^n), \quad P_c(z)=0\}.
\end{equation}

\begin{lemma}\label{equiv}
If $0< |c|<\sqrt2$, then the $H^2$-norm is equivalent to the norm $||\cdot||$ defined by
\begin{equation}\label{norm}
||u||^2 = \int_{\R^n} (\Delta u)^2 - (c\cdot \nabla u)^2 + u^2 \:dx.
\end{equation}
\end{lemma}

\begin{proof}
See \cite[Lemma 2.2]{CheMcK} for the case $n=1$; the general case is almost identical.
\end{proof}

\begin{lemma}\label{usp}
Assume (A1)-(A2) and that $0< |c|<\sqrt2$.  Then there exists some $0\nequiv u\in H^2$ such that $P_c(u)=0$.
\end{lemma}

\begin{proof}
First of all, let us note that there exists some $0\nequiv w\in H^2$ such that
\begin{equation}\label{spe}
||w||^2 - ||w||_{L^2(\R^n)}^2 = \int_{\R^n} (\Delta w)^2 - (c\cdot \nabla w)^2 \:dx < 0.
\end{equation}
In fact, $w(x)= -e^{-\al |x|^2}$ is such for all small enough $\al>0$ because
\begin{equation*}
||w||^2 - ||w||_{L^2(\R^n)}^2 =  \al \Bigl( n(n+2)\al - |c|^2 \Bigr) \cdot \left( \frac{\pi}{2\al} \right)^{n/2},
\end{equation*}
as one can easily check.  We now fix $\al>0$ so that \eqref{spe} holds and we focus on
\begin{equation}\label{Pcaw}
P_c(Aw) = A^2 ||w||^2 + \int_{\R^n} Aw f(Aw) \:dx, \quad\quad A>0.
\end{equation}
By our assumption (A2) and Sobolev embedding, the integral term is bounded by
\begin{equation*}
\left| \int_{\R^n} Aw f(Aw) \:dx \right| \leq C_1 A^{q+1} \int_{\R^n} |w(x)|^{q+1} \:dx \leq C_2A^{q+1} \,||w||^{q+1}.
\end{equation*}
Since $q>1$, this implies $P_c(Aw)>0$ for all small enough $A>0$, so it suffices to check that $P_c(Aw)<0$ for all large enough $A>0$.
Let us now use \eqref{spe} and \eqref{Pcaw} to get
\begin{align*}
\lim_{A\to\infty} \frac{P_c(Aw)}{A^2}
&= ||w||^2 -||w||_{L^2(\R^n)}^2 + \lim_{A\to\infty} \int_{\R^n} \frac{Aw f(Aw) + A^2w^2}{A^2} \:dx \\
&< \lim_{A\to\infty} \int_{\R^n} \frac{Aw f(Aw) + A^2w^2}{A^2} \:dx.
\end{align*}
According to our assumption (A1), there exist constants $u_0<0$ and $\b\in\R$ such that
\begin{equation*}
uf(u) + u^2 \leq \b u\quad\text{for all $u\leq u_0<0$.}
\end{equation*}
Since $-1\leq w(x)\leq 0$ for all $x\in\R^n$ by definition, this also implies
\begin{align*}
\lim_{A\to\infty} \frac{P_c(Aw)}{A^2}
&<\lim_{A\to\infty} \int_{-A\leq Aw\leq u_0} \frac{\b w}{A} \:dx +\int_{u_0\leq Aw\leq 0} \frac{Aw f(Aw) +A^2w^2}{A^2} \:dx.
\end{align*}
Since $w$ is integrable and $f(u)+u$ is bounded on $[u_0,0]$, the right hand side is obviously zero.  In particular, $P_c(Aw)< 0$ for
all large enough $A$ and the result follows.
\end{proof}

\begin{lemma}\label{2min}
Assume (A1)-(A3) and let
\begin{equation}\label{Jc}
J_c(z) = I_c(z) - \frac{1}{\mu} P_c(z)= \frac{\mu-2}{2\mu} \,||z||^2 + \int_{\R^n} G(z) \:dx
\end{equation}
for each $z\in H^2(\R^n)$.  Then a minimizer exists for
\begin{equation}\label{dc}
d_c = \inf \{ I_c(w) : 0\nequiv w\in H^2, \quad P_c(w)=0\},
\end{equation}
provided that a minimizer exists for
\begin{equation}\label{mc}
m_c = \inf \{ J_c(w) : 0\nequiv w\in H^2, \quad P_c(w)\leq 0\}.
\end{equation}
\end{lemma}

\begin{proof}
Suppose that $u$ is a minimizer for $m_c$, in which case
\begin{equation*}
0\nequiv u\in H^2, \quad\quad P_c(u)\leq 0, \quad\quad J_c(u)= m_c.
\end{equation*}
We note that $P_c(\al u)>0$ for all small enough $\al>0$, as our assumption (A2) gives
\begin{align}
P_c(\al u) = \al^2 ||u||^2 + \int_{\R^n} \al uf(\al u) \:dx
&\geq \al^2 ||u||^2 - C_1\al^{q+1} ||u||_{L^{q+1}}^{q+1} \label{use1} \\
&\geq \al^2 ||u||^2 - C\al^{q+1} ||u||^{q+1} \notag
\end{align}
with $q>1$.  In particular, $P_c(\al u)=0$ for some $0<\al \leq 1$ and we easily get
\begin{align*}
d_c \leq I_c(\al u) = J_c(\al u)
&= \frac{\mu-2}{2\mu} \,\al^2 ||u||^2 + \int_{\R^n} G(\al u) \:dx.
\end{align*}
In view of our assumption (A3), $\mu\geq 2$ and $G\geq 0$ is convex with $G(0)=0$, hence
\begin{align}\label{au}
d_c \leq I_c(\al u)
&\leq \frac{\mu-2}{2\mu} \,\al^2 ||u||^2 + \al \int_{\R^n} G(u) \:dx \notag \\
&\leq \frac{\mu-2}{2\mu} \,||u||^2 + \int_{\R^n} G(u) \:dx \\
&= J_c(u) =m_c. \notag
\end{align}
Since $m_c\leq d_c$ by definitions \eqref{dc} and \eqref{mc}, this means $\al u$ is a minimizer for $d_c$.
\end{proof}

\begin{rem}
Under the assumptions of Theorem \ref{exi}, one can easily improve the conclusions of Lemma \ref{2min} in the sense that $u$ is a
minimizer for $d_c$ if and only if it is a minimizer for $m_c$.  This amounts to showing that \eqref{au} can only hold with equality
when $\al=1$.  Nevertheless, we shall not need a stronger version of this lemma in what follows.
\end{rem}

\section{Proof of our main result} \label{pomr}
In this section, we shall prove our main result, Theorem \ref{exi}.  In fact, we shall also give a characterization of ground state
solutions, as we do in the following theorem.

\begin{theorem}\label{new}
Assume (A1)-(A4) and that $0<|c|<\sqrt2$.
\begin{itemize}
\item[(a)]
If $\mu>2$, then equation \eqref{4} has ground state solutions for any $n\geq 1$.
\item[(b)]
If $\mu=2$, then equation \eqref{4} has ground state solutions when $n=1,2,3$.
\end{itemize}
In any case, however, $w$ is a ground state of \eqref{4} if and only if $w$ is a minimizer for $d_c$.
\end{theorem}

Recall that solutions of \eqref{4} correspond to critical points of the functional $I_c$ in \eqref{Ic}.  We shall henceforth denote by
$S_c$ the set of all nontrivial critical points, namely
\begin{equation}\label{Sc}
S_c = \{ u\in H^2: u\nequiv 0,\quad I_c'(u)=0\},
\end{equation}
and we shall also denote by $G_c$ the set of all ground states, namely
\begin{equation}\label{Gc}
G_c = \{ w\in S_c: I_c(w)\leq I_c(u) \quad\text{for all $u\in S_c$} \}.
\end{equation}
In order to prove Theorem \ref{new}, we need to make use of the following two lemmas.  We refer the reader to \cite{FrLiLo} for the
first one and \cite{Lieb} for the second; although the second was originally stated for sequences in $H^1$, the proof in \cite{Lieb}
applies verbatim to yield a similar statement for sequences in $H^2$.

\begin{lemma}
Suppose $f$ is a measurable function on $\R^n$ which satisfies
\begin{equation*}
||f||_{L^\al}\leq C_\al, \quad\quad ||f||_{L^\b} \geq C_\b, \quad\quad ||f||_{L^\gamma} \leq C_\gamma
\end{equation*}
for some $1<\al<\b<\gamma<\infty$ and some $C_\al,C_\b,C_\gamma>0$.  Then there exist some fixed constants $\e,C_0>0$ such that the
Lebesgue measure of $\{x\in\R^n: |f(x)|\geq \e\}$ is at least $C_0$.
\end{lemma}

\begin{lemma}
Suppose $\{f_k\}\subset H^2(\R^n)$ is a uniformly bounded sequence and suppose there exist some fixed $\e,C_0>0$ such
that the Lebesgue measure of $\{x\in\R^n: |f_k(x)|\geq \e\}$ is at least $C_0$ for each $k$.  Then there exists a
sequence of points $\{x_k\}\subset \R$ such that $F_k(y)= f_k(y+x_k)$ has a subsequence which converges weakly in $H^2$
to some nonzero function $F\in H^2$.
\end{lemma}

\begin{proof_of}{Theorem \ref{new}}
We divide the argument into several steps.

\Step{1} (Weak convergence). Let $\{z_k\}\subset H^2$ be a minimizing sequence for $m_c$ so that
\begin{equation}\label{mseq}
0\nequiv z_k\in H^2, \quad\quad P_c(z_k)\leq 0, \quad\quad \lim_{k\to\infty} J_c(z_k) = m_c.
\end{equation}
Since $J_c$ is non-negative by definition \eqref{Jc}, one has $0\leq m_c\leq J_c(u)=I_c(u)$, where $u$ is the function of
Lemma \ref{usp}.  In particular, $m_c$ is bounded and the same is true for each
\begin{equation*}
J_c(z_k)= \frac{\mu-2}{2\mu} \,||z_k||^2 + \int_{\R^n} G(z_k) \:dx.
\end{equation*}
When $\mu>2$, this already implies that $\{z_k\}$ is uniformly bounded in $H^2$, as $G\geq 0$ by (A3).  When $\mu=2$, we reach the
same conclusion using our additional assumption in (A3), as
\begin{align*}
||z_k||^2 &= P_c(z_k) - \int_{\R^n} z_k f(z_k) \:dx \\
&\leq C_2^{-1} ||z_k||_{L^\infty} \int_{\R^n} G(z_k)\:dx \\
&\leq C||z|| \cdot J_c(z_k)
\end{align*}
by Sobolev embedding when $n=1,2,3$.  This shows that $\{z_k\}$ is uniformly bounded in $H^2$ for each of the two cases considered in
Theorem \ref{new}.  Moreover, \eqref{use1} gives
\begin{align*}
0\geq P_c(z_k) \geq ||z_k||^2 - C_1 ||z_k||_{L^{q+1}}^{q+1} \geq C_3||z_k||_{L^{q+1}}^2 - C_1||z_k||_{L^{q+1}}^{q+1}
\end{align*}
so the norms $||z_k||_{L^{q+1}}$ are uniformly bounded from below.  Using the previous two lemmas and replacing $\{z_k\}$ by a
subsequence, we conclude that there is a nonzero function $z\in H^2$ such that $z_k\to z$ weakly in $H^2$.  By compactness, we may
additionally assume
\begin{equation}\label{stro}
z_k\to z \quad\text{strongly in $L^{p+1}_\text{loc}(\R^n)$.}
\end{equation}

\Step{2} (Two limits). To show that $z$ is a minimizer for $m_c$, we shall need to know that
\begin{align}
\lim_{k\to\infty} P_c(z_k) - P_c(z_k - z) - P_c(z) &= 0 \label{lim1}, \\
\lim_{k\to\infty} J_c(z_k) - J_c(z_k - z) - J_c(z) &= 0 \label{lim2}.
\end{align}
In order to prove these two facts, let us first recall our definitions
\begin{align*}
P_c(z) = ||z||^2 + \int_{\R^n} zf(z) \:dx, \quad\quad
&J_c(z)= \frac{\mu-2}{2\mu} \,||z||^2 + \int_{\R^n} F(z) - \frac{1}{\mu} \,zf(z) \:dx.
\end{align*}
Since $z_k\to z$ weakly in $H^2$ by above, the expression
\begin{equation*}
||z_k||^2 - ||z_k - z||^2 - ||z||^2 = 2(z, z_k-z)
 \end{equation*}
does go to zero as $k\to\infty$.  To prove \eqref{lim1} and \eqref{lim2}, it remains to check that
\begin{equation}\label{che}
\lim_{k\to\infty} \int_{\R^n} H(z_k) - H(z_k -z) -H(z) \:dx = 0
\end{equation}
when $H(u)= F(u)$ or $H(u)=uf(u)$.  In either of these two cases, \eqref{f1} gives
\begin{equation*}
|H(u)-H(v)|\leq C(|u|^p +|v|^p)  \,|u-v| \quad\text{for all $u,v\in \R$.}
\end{equation*}
Given an arbitrary set $\Omega\subset \R^n$, we then easily find that
\begin{align}\label{Hes}
\int_\Omega |H(u)-H(v)| \:dx
&\leq C||u-v||_{L^{p+1}(\Omega)} \cdot \Bigl( ||u||_{L^{p+1}(\R^n)}^p + ||v||_{L^{p+1}(\R^n)}^p \Bigr) \notag \\
&\leq C||u-v||_{L^{p+1}(\Omega)} \cdot \Bigl( ||u||^p + ||v||^p \Bigr)
\end{align}
by H\"older's inequality and Sobolev embedding.  We now use this fact to prove \eqref{che}.

Let $\e>0$ be given and fix some compact set $A\subset \R^n$ which is large enough that
\begin{equation}\label{ext}
||z||_{H^2(\R^n \backslash A)} \leq \e.
\end{equation}
Applying our general estimate \eqref{Hes} for various choices of $u,v$ and $\Omega$, we get
\begin{align*}
\int_A |H(z_k) -H(z)| + |H(z_k-z)| \:dx
&\leq C||z_k-z||_{L^{p+1}(A)} \cdot \Bigl( ||z_k||^p + ||z||^p \Bigr), \\
\int_{\R^n \backslash A} |H(z_k) -H(z_k-z)| + |H(z)| \:dx
&\leq C||z||_{L^{p+1}(\R^n \backslash A)} \cdot \Bigl( ||z_k||^p + ||z||^p \Bigr).
\end{align*}
As the norms $||z_k||$ are uniformly bounded by Step 1, this actually gives
\begin{align*}
\int_{\R^n} |H(z_k) - H(z_k -z) -H(z)| \:dx
&\leq C||z_k-z||_{L^{p+1}(A)} + C||z||_{H^2(\R^n \backslash A)}\\
&\leq C\e
\end{align*}
for large $k$ by \eqref{stro} and \eqref{ext}.  Since $\e>0$ was arbitrary, the desired \eqref{che} follows.

\Step{3} (Existence of a minimizer). We shall now show that $P_c(z)\leq 0$ and that $J_c(z)=m_c$. Suppose $P_c(z)>0$ for
the sake of contradiction.  Since $P_c(z_k)\leq 0$ for all $k$, we get
\begin{equation}\label{Pcla}
P_c(z_k-z) < 0 \quad\text{for all large enough $k$}
\end{equation}
by \eqref{lim1}.  Using the definition of $m_c$ together with \eqref{mseq} and \eqref{lim2}, we conclude that
\begin{equation*}
J_c(z)= \lim_{k\to\infty} J_c(z_k) - J_c(z_k-z) = m_c - \lim_{k\to\infty} J_c(z_k-z) \leq 0.
\end{equation*}
Since $J_c\geq 0$ by definition, however, equality must hold in the last inequality and thus
\begin{equation*}
\lim_{k\to\infty} J_c(z_k-z) = m_c.
\end{equation*}
Together with \eqref{Pcla}, this means $\{z_k-z\}$ becomes a minimizing sequence for $m_c$, if a finite number of its terms are
ignored.  By Step 1, such a sequence must have a nonzero weak limit up to a subsequence, a contradiction that gives $P_c(z)\leq 0$.
Next, we note that
\begin{equation*}
J_c(u)= \frac{\mu-2}{2\mu} \,||u||^2 + \int_{\R^n} G(u) \:dx
\end{equation*}
is weakly lower semi-continuous by Fatou's lemma and (A3).  In particular,
\begin{equation*}
J_c(z) \leq \liminf_{k\to\infty} J_c(z_k) = m_c
\end{equation*}
and so equality must hold in the last inequality because $P_c(z)\leq 0$ by above.

\Step{4} (Ground states). Since a minimizer is now known to exist for $m_c$, Lemma \ref{2min} shows that one
exists for $d_c$.  To see that every minimizer for $d_c$ is a ground state, suppose
\begin{equation*}
I_c(w) = d_c = \min \{ I_c(z) : 0\nequiv z\in H^2, \quad P_c(z)=0\}.
\end{equation*}
Then there exists a Lagrange multiplier $\la\in\R$ such that $I_c'(w)= \la P_c'(w)$ and thus
\begin{equation}\label{eq1}
\la \br{P_c'(w),w} = \br{I_c'(w),w} = P_c(w) = 0
\end{equation}
by \eqref{Pc}.  Once we now compute the Fr\'echet derivative of \eqref{Pc}, we find that
\begin{align}
\br{P_c'(w),w}
&= 2||w||^2 + \int_{\R^n} wf(w) \:dx + \int_{\R^n} w^2f'(w) \:dx \notag \\
&= 2P_c(w) + \int_{\R^n} w^2f'(w) - wf(w) \:dx \notag \\
&= \int_{\R^n} w^2f'(w) - wf(w) \:dx. \label{strict}
\end{align}
In view of our assumption (A4), the integral is non-positive and can only be zero, if
\begin{equation*}
\int_{\R^n} wf(w) \:dx = P_c(w) - ||w||^2 = -||w||^2
\end{equation*}
itself is zero.  Since that is not the case, however, the integral in \eqref{strict} is strictly negative, so $\la=0$ by
\eqref{eq1} and $I_c'(w)= \la P_c'(w)=0$ by above; in other words, $w\in S_c$ is a nontrivial solution of \eqref{pde}. Given any
other nontrivial solution $u\in S_c$, we easily get
\begin{equation*}
P_c(u)= \br{I_c'(u),u}=0\so d_c\leq I_c(u)
\end{equation*}
by \eqref{Pc} and \eqref{dc}.  Since $w$ is a minimizer for $d_c$, this actually shows that $w\in G_c$.

\Step{5} (End of proof). We now show that each ground state $w\in G_c$ is a minimizer for $d_c$. According to our
definition \eqref{Gc}, we have $P_c(w)=\br{I_c'(w),w}= 0$ and so
\begin{equation*}
I_c(w) \geq d_c = I_c(u)
\end{equation*}
for any minimizer $u$.  As $w\in G_c$, on the other hand, we also have $I_c(w)\leq I_c(u)$ by \eqref{Gc}.  In particular,
equality holds in the last inequality and $w$ itself is a minimizer for $d_c$.
\end{proof_of}

\end{document}